\documentclass[11pt]{amsart}
\usepackage[margin=1in]{geometry}
\usepackage{latexsym}

\usepackage[utf8]{inputenc}
\usepackage{amsfonts}
\usepackage{amsmath}
\usepackage{amssymb}
\usepackage{amsthm}
\usepackage{array}
\usepackage{commath}
\usepackage{subfig}
\usepackage{float}
\usepackage{mathrsfs}
\usepackage[colorlinks, breaklinks=true]{hyperref}
\usepackage[margin=0.5cm]{caption}
\hypersetup{
	linkcolor=blue,
	citecolor=green,
	filecolor=Red,
	urlcolor=NavyBlue,
	menucolor=Red,
	runcolor=Red}
\usepackage{thmtools}
\usepackage{thm-restate}
\usepackage{cleveref}
\usepackage{appendix}
\newcommand{\PreserveBackslash}[1]{\let\temp=\\#1\let\\=\temp}
\newcolumntype{C}[1]{>{\PreserveBackslash\centering}p{#1}}
\newcolumntype{R}[1]{>{\PreserveBackslash\raggedleft}p{#1}}
\newcolumntype{L}[1]{>{\PreserveBackslash\raggedright}p{#1}}

\numberwithin{equation}{section}
\allowdisplaybreaks

\usepackage[dvipsnames]{xcolor}

\usepackage[style=alphabetic, giveninits, maxbibnames=9]{biblatex}
\renewbibmacro{in:}{}
\DeclareFieldFormat[article]{title}{#1}

\usepackage{todonotes}
\usepackage{amsmath}

\newcommand{\E}{\mathbb{E}}
\newcommand{\Var}{\textup{Var}}

\newcommand{\FS}{\mathsf{FS}}

\newcommand{\Path}{\textsf{\textup{Path}}}

\newcommand{\rev}{\textup{rev}}

\newtheorem{theorem}{Theorem}[section]
\newtheorem{lemma}[theorem]{Lemma}
\newtheorem{proposition}[theorem]{Proposition}
\newtheorem{corollary}[theorem]{Corollary}

\newtheorem{question}[theorem]{Question}
\newtheorem{problem}[theorem]{Problem}

\theoremstyle{definition}
\newtheorem{definition}[theorem]{Definition}

\title{Bipartite Friends and Strangers Walking on Bipartite Graphs}
\author{Ryan Jeong}
\address{Department of Pure Mathematics and Mathematical Statistics, University of Cambridge, Wilberforce Road, Cambridge, CB3 0WA, UK}
\email{rj450@cam.ac.uk}

\begin{document}

\maketitle

\begin{abstract}
    Given $n$-vertex simple graphs $X$ and $Y$, the friends-and-strangers graph $\FS(X, Y)$ has as its vertices all $n!$ bijections from $V(X)$ to $V(Y)$, where two bijections are adjacent if and only if they differ on two adjacent elements of $V(X)$ whose mappings are adjacent in $Y$. We consider the setting where $X$ and $Y$ are both edge-subgraphs of $K_{r,r}$: due to a parity obstruction, $\FS(X,Y)$ is always disconnected in this setting. Modestly improving a result of Bangachev, we show that if $X$ and $Y$ respectively have minimum degrees $\delta(X)$ and $\delta(Y)$ and they satisfy $\delta(X) + \delta(Y) \geq \lfloor 3r/2 \rfloor + 1$, then $\FS(X,Y)$ has exactly two connected components. This proves that the cutoff for $\FS(X,Y)$ to avoid isolated vertices is equal to the cutoff for $\FS(X,Y)$ to have exactly two connected components. We also consider a probabilistic setup in which we fix $Y$ to be $K_{r,r}$, but randomly generate $X$ by including each edge in $K_{r,r}$ independently with probability $p$. Invoking a result of Zhu, we exhibit a phase transition phenomenon with threshold function $(\log r)/r$. More precisely, below the threshold, $\FS(X,K_{r,r})$ has more than two connected components with high probability, while above the threshold, $\FS(X,K_{r,r})$ has exactly two connected components with high probability. Altogether, our results settle a conjecture and completely answer two problems of Alon, Defant, and Kravitz.
\end{abstract}

\section{Introduction}

\subsection{Background and Motivation}

Let $X$ and $Y$ be $n$-vertex simple graphs. Interpret the vertices of $X$ as positions, and the vertices of $Y$ as people. Two people in the vertex set of $Y$ are friends if they are adjacent and strangers if they are not. Each person chooses a position, producing a starting configuration. From here, at any point in time, two friends standing on adjacent positions may switch places: we call this operation a friendly swap. Our main interest in this paper will be to understand, in terms of assumptions on the structure of the graphs $X$ and $Y$, which configurations are reachable from which other configurations via some sequence of friendly swaps.

We may formalize this setup using the following definition, illustrated in Figure \ref{fig:defn}.

\begin{definition}[\cite{defant2021friends}] \label{defn:fs_def}
Let $X$ and $Y$ be simple graphs on $n$ vertices. The \textit{\textcolor{red}{friends-and-strangers graph}} of $X$ and $Y$, denoted $\FS(X, Y)$, is a graph with vertices consisting of all bijections from $V(X)$ to $V(Y)$, with bijections $\sigma, \tau \in \FS(X,Y)$ adjacent if and only if there exists an edge $\{a, b\}$ in $X$ such that
\begin{enumerate}
    \item $\{\sigma(a), \sigma(b)\} \in E(Y)$,
    \item $\sigma(a) = \tau(b), \ \sigma(b) = \tau(a)$,
    \item $\sigma(c) = \tau(c)$ for all $c \in V(X) \setminus \{a, b\}$.
\end{enumerate}
In other words, $\sigma$ and $\tau$ differ on two adjacent vertices of $X$ whose images under $\sigma$ are adjacent in $Y$. For any such bijections $\sigma, \tau$, we say that $\tau$ is reached from $\sigma$ by an \textit{\textcolor{red}{$(X, Y)$-friendly swap}}.
\end{definition}

\begin{figure}[ht]
\begin{minipage}{.5\linewidth}
\centering
\subfloat[The graph $X$.]{\label{fig:X_gph}\includegraphics[width=0.48\textwidth]{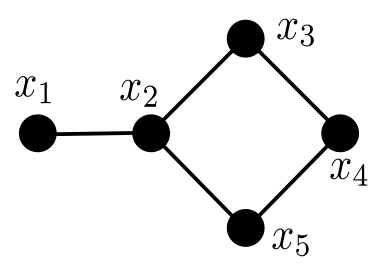}}
\end{minipage}%
\hfill
\begin{minipage}{.5\linewidth}
\centering
\subfloat[The graph $Y$.]{\label{fig:Y_gph}\includegraphics[width=0.48\textwidth]{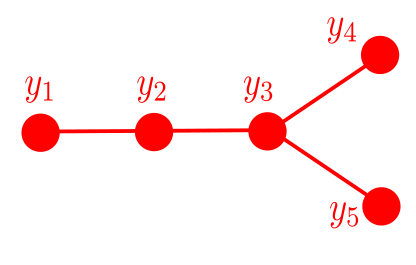}}
\end{minipage}\par\medskip
\centering
\subfloat[A sequence of $(X, Y)$-friendly swaps. Transpositions between adjacent configurations denote the two vertices in $X$ over which the $(X, Y)$-friendly swap takes place. Vertices in $Y$ (in red text) are placed upon vertices of $X$ (in black text). The leftmost configuration corresponds to the bijection $\sigma$ such that $\sigma(x_1) = y_1$, $\sigma(x_2) = y_5$, $\sigma(x_3) = y_3$, $\sigma(x_4) = y_4$, and $\sigma(x_5) = y_2$. Configurations correspond to vertices in $\FS(X, Y)$.]{\label{fig:friendly_swaps}\includegraphics[width=.99\textwidth]{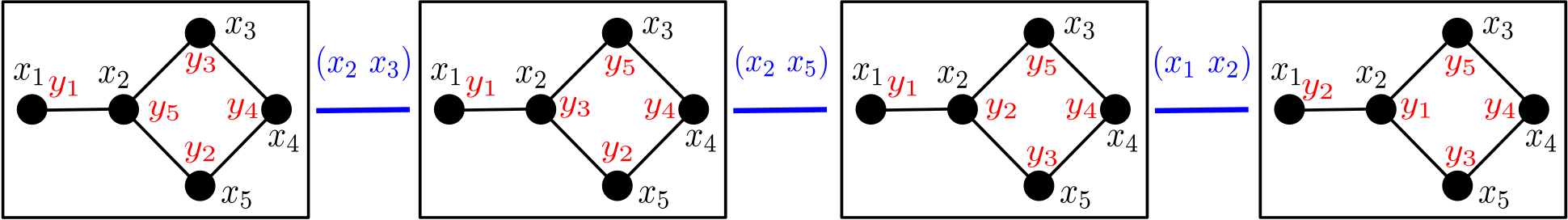}}
\caption{A sequence of $(X, Y)$-friendly swaps in $\FS(X, Y)$. Configurations in the bottom row are vertices in $V(\FS(X, Y))$. Two consecutive configurations differ by an $(X, Y)$-friendly swap, so the corresponding vertices are adjacent in $\FS(X, Y)$. The figure and subcaptions are adapted from \cite{jeong2023diameters}.}
\label{fig:defn}
\end{figure}

Since Defant and Kravitz \cite{defant2021friends} introduced friends-and-strangers graphs, their study has been a productive direction of research. Indeed, Definition \ref{defn:fs_def} lends itself to several natural directions of inquiry. One such direction is to assume $X$ to be some highly structured graph, and study the structure of $\FS(X,Y)$ for arbitrary graphs $Y$: see \cite{defant2022connectedness, jeong2022structural, lee2022connectedness, milojevic2023connectivity, naatz2000graph, wang2023connectivity, wang2022lollipop, zhu2023evacuating}. It is also very natural to ask extremal and probabilistic questions about friends-and-strangers graphs: in addition to the present paper, see \cite{alon2022typical, bangachev2022asymmetric, jeong2023diameters, milojevic2023connectivity, wang2022connectedness}. We also mention that many other works in combinatorics and theoretical computer science may be recast using the language of friends-and-strangers. We illustrate with a non-exhausting listing of examples. Studying the famous $15$-puzzle is equivalent to studying $\FS(X,Y)$ where we let $X$ be the $4$-by-$4$ grid and $Y$ be a star graph; see \cite{brunck2023books, demaine2018simple, parberry2015solving, wilson1974graph, yang2011sliding} for generalizations. The works \cite{barrett1999elements, naatz2000graph, reidys1998acyclic, stanley2012equivalence} all investigate the structure of the graph $\FS(\Path_n, Y)$ for certain graphs $Y$. Asking if $X$ and $Y$ pack \cite{bollobas1978packings, bollobas2017packing, kuhn2009minimum, sauer1978edge, yap1988packing, yuster2007combinatorial} is equivalent to asking if there exists an isolated vertex in $\FS(X,Y)$, while the token swapping problem \cite{aichholzer2021hardness, biniaz2023token, bonnet2018complexity, miltzow_et_al:LIPIcs:2016:6408, yamanaka2015swapping} on the graph $X$ corresponds to studying distances between configurations in $\FS(X, K_n)$. We refer the reader to \cite[Subsection 1.1]{jeong2023diameters} for additional connections to algebraic combinatorics and statistical physics.

As suggested in the first paragraph, however, the most fundamental issue concerning friends-and-strangers graphs that one can study is their connectivity. Under what conditions on $X$ and $Y$ will $\FS(X,Y)$ be connected? If we proceed under a regime in which $\FS(X,Y)$ cannot be connected, how small can the number of connected components get, and what conditions on $X$ and $Y$ ensure that we achieve the smallest possible number of connected components? Of course, the resolution of these questions depends on the assumptions on $X$ and $Y$ under which we work. In this paper, we assume that $X$ and $Y$ are both edge-subgraphs of a complete bipartite graph whose partite sets have the same size $r$, and investigate what happens as $r$ grows large. Part of the motivation for studying this setting comes from the observation that if $X$ and $Y$ are both bipartite, $\FS(X,Y)$ cannot be connected; see the discussion around \cite[Proposition 2.7]{defant2021friends} and \cite[Subsection 2.3]{alon2022typical} for a parity obstruction which demonstrates why this is the case.\footnote{This is for reasons akin to why the so-called $14$-$15$ puzzle is unsolvable.} Indeed, this particular setup was of interest to many: it was also studied in \cite{alon2022typical, bangachev2022asymmetric, milojevic2023connectivity, wang2022connectedness, zhu2023evacuating}.

\subsection{Notation and Conventions}

We assume in this paper that all graphs are simple. For the sake of completeness, we list the following standard conventions, which we make use of throughout the article. For a graph $G$,
\begin{itemize}
    \item we let $V(G)$ and $E(G)$ respectively denote the vertex and edge sets of $G$;
    \item if $S \subseteq V(G)$, then we let $G|_S$ denote the induced subgraph of $G$ on $S$;
    \item if $v \in V(G)$, then we let $N(v) = \{u \in V(G) : \{u,v\} \in E(G)\}$ denote the (open) neighborhood of $v$;
    \item we let $\delta(G)$ and $\Delta(G)$ respectively denote the minimum degree and maximum degree of $G$;
    \item we let $\mathcal G(G, p)$ denote the probability space of edge-subgraphs of $G$ in which each edge appears independently with probability $p$.
\end{itemize}
We let $K_{r,r}$ denote the complete bipartite graph whose two partite sets both have size $r$. If $\Sigma$ is a finite sequence, then $\rev(\Sigma)$ denotes the reverse of $\Sigma$. Finally, for two sets $S$ and $T$, we let $S \bigtriangleup T = (S \setminus T) \cup (T \setminus S)$ denote their symmetric difference.

\subsection{Main Results}

As previously mentioned, when $X$ and $Y$ are both edge-subgraphs of $K_{r,r}$, the best we can hope for is that $\FS(X,Y)$ has exactly two connected components. A natural extremal problem results from asking for conditions on the minimum degrees of $X$ and $Y$ ensuring that $\FS(X,Y)$ has exactly two connected components. In this direction, we let $d_{r,r}$ be the smallest nonnegative integer such that whenever $X$ and $Y$ are edge-subgraphs of $K_{r,r}$ with $\delta(X) \geq d_{r,r}$ and $\delta(Y) \geq d_{r,r}$, $\FS(X,Y)$ has exactly two connected components. This problem was first studied in \cite[Sections 5 and 6]{alon2022typical}, which proved bounds on $d_{r,r}$ that were tight up to additive constants. Asymmetrizing the problem by dropping the assumption that $X$ and $Y$ must satisfy the same minimum degree condition, \cite[Sections 6 and 7]{bangachev2022asymmetric} generalized the results of \cite{alon2022typical}, again obtaining bounds that were tight up to additive constants. In the present article, we sharpen the results of \cite{bangachev2022asymmetric} by shaving off the additive constants, producing completely tight conditions concerning when $\FS(X,Y)$ has exactly two connected components. 

In a different direction (and in the spirit of prior work on the topic, as mentioned earlier), instead of varying both edge-subgraphs $X$ and $Y$, we may fix $Y = K_{r,r}$, and ask for conditions on $X$ which ensure that $\FS(X,K_{r,r})$ has exactly two connected components. A stochastic analogue of this problem is obtained by letting $X \in \mathcal G(K_{r,r},p)$ and by asking for both conditions on $p = p(r)$ which ensure that $\FS(X,K_{r,r})$ has exactly two connected components with high probability (that is, with probability tending to $1$ as $r \to \infty$) and for conditions on $p$ which ensure that $\FS(X,K_{r,r})$ has more than two connected components with high probability. This problem was raised in \cite[Question 7.6]{alon2022typical}. Here, we invoke a recent result of \cite{zhu2023evacuating}, which quickly leads to a complete answer this problem. Specifically, we find a phase transition with threshold function $p(r) = (\log r)/r$.

We now state the specific results that we will prove. In Section \ref{sec:minimum_degree}, we prove the following result.
\begin{theorem} \label{thm:minimum_degree}
    Let $r \geq 4$, and let $X$ and $Y$ be edge-subgraphs of $K_{r,r}$ such that
    \begin{align*}
        \delta(X) + \delta(Y) \geq \lfloor 3r/2 \rfloor + 1.
    \end{align*}
    Then $\FS(X,Y)$ has exactly two connected components.
\end{theorem}

Theorem \ref{thm:minimum_degree} sharpens \cite[Theorem 1.10]{bangachev2022asymmetric}, which gave a lower bound of $3r/2 + 1$. Together with \cite[Theorem 1.11]{bangachev2022asymmetric} and a computer check for the $r=3$ case, Theorem \ref{thm:minimum_degree} in the settings $\delta(X) = \delta(Y)$ and $\delta(Y) = r$ respectively implies the following statements.

\begin{corollary} \label{cor:symmetric_minimum_degree}
    We have $d_{r,r} = \lceil (3r+1)/4 \rceil$.
\end{corollary}

\begin{corollary} \label{cor:complete_bipartite_minimum_degree}
    For each $r \geq 2$, let $d_{r,r}^*$ be the smallest nonnegative integer such that whenever $X$ is an edge-subgraph of $K_{r,r}$ with $\delta(X) \geq d_{r,r}^*$, $\FS(X,Y)$ has exactly two connected components. We have that
    \begin{align*}
        d_{r,r}^* = \begin{cases}
            \left\lfloor r/2 \right\rfloor + 1 & r \neq 3, \\
            3 & r = 3.
        \end{cases}
    \end{align*}
\end{corollary}

Corollary \ref{cor:symmetric_minimum_degree} settles \cite[Conjecture 7.4]{alon2022typical}, while Corollary \ref{cor:complete_bipartite_minimum_degree} sharpens \cite[Corollary 1.12]{bangachev2022asymmetric} and provides a complete answer to \cite[Problem 7.7]{alon2022typical}. Though the improvement that Theorem \ref{thm:minimum_degree} provides over \cite[Theorem 1.10]{bangachev2022asymmetric} is very modest, we recall the macroscopic intent of \cite[Conjecture 7.4]{alon2022typical}, as seen in the discussion around \cite[Theorem 1.4]{alon2022typical}: determining whether or not the cutoff for avoiding isolated vertices is the same as the cutoff for $\FS(X,Y)$ to have the smallest possible number of connected components. It follows from the proof of \cite[Theorem 1.11]{bangachev2022asymmetric} that $\lfloor 3r/2 \rfloor+1$ is the cutoff to avoid isolated vertices, so Theorem \ref{thm:minimum_degree} confirms that if $X$ and $Y$ are both edge-subgraphs of $K_{r,r}$, these cutoffs are exactly the same for all $r \geq 4$. Our results here may thus be interpreted as providing further motivation for \cite[Question 7.5]{alon2022typical}, which asks whether there exists an analogue for friends-and-strangers graphs of the well-known phenomenon that for a random graph process, the stopping times for no isolated vertices and the graph being connected are asymptotically almost surely the same. Indeed, Theorem \ref{thm:minimum_degree} might be thought of as giving something of a deterministic, extremal variant for the bipartite setting.

In Section \ref{sec:random_edge_subgraphs}, we prove the following result.
\begin{theorem} \label{thm:random_edge_subgraphs}
    Let $X$ be a random graph in $\mathcal G(K_{r,r},p)$, where $p=p(r)$ depends on $r$. Let $\omega(r)$ be a function of $r$ such that $\omega(r) \to \infty$. If
    \begin{align*}
        p = \frac{\log r - \omega(r)}{r}
    \end{align*}
    and $p \geq 0$, then $\FS(X, K_{r,r})$ has more than two connected components with high probability. If 
    \begin{align*}
        p = \frac{\log r + \omega(r)}{r},
    \end{align*}
    and $p \leq 1$, then $\FS(X,K_{r,r})$ has exactly two connected components with high probability.
\end{theorem}
Theorem \ref{thm:random_edge_subgraphs}, which identifies $p(r) = (\log r)/r$ as a threshold function, identifies a phase transition and provides an essentially complete answer to \cite[Question 7.6]{alon2022typical}.\footnote{The arXiv version of \cite[Question 7.6]{alon2022typical}, but not the journal version, contains a mistake in its statement. Specifically, if $X$ is a random graph in $\mathcal G(K_{r,r},p)$, the arXiv version asks for conditions on $p$ ensuring that $\FS(X, K_{r,r})$ is disconnected with high probability and conditions on $p$ ensuring that $\FS(X, K_{r,r})$ is connected with high probability. From our discussion, that problem is trivial, since $\FS(X, K_{r,r})$ is always disconnected in this setting when $r \geq 2$.} We conclude in Section \ref{sec:future_directions} with some suggested directions for future research.

\section{Preliminaries} \label{sec:preliminaries}

In this section, we list results from prior work that will be relevant later in the paper. We mention that some of the results below are special cases of what is stated in the corresponding cited result.

\begin{proposition}[{\cite[Proposition 2.6]{defant2021friends}}] \label{prop:basic_properties}
Definition \ref{defn:fs_def} is symmetric with respect to $X$ and $Y$: if $X$ and $Y$ are both $n$-vertex graphs, we have that $\FS(X,Y) \cong \FS(Y,X)$. 
\end{proposition}

The following Proposition \ref{prop:two_complete_bipartite} presents an obstruction on the connectivity of $\FS(X,Y)$ when $X$ and $Y$ are edge-subgraphs of $K_{r,r}$, and also shows that the smallest number of connected components that $\FS(X,Y)$ may have in this setting is two.

\begin{proposition}[{\cite[Proposition 2.6]{alon2022typical}}] \label{prop:two_complete_bipartite}
    For $r \geq 2$, $\FS(K_{r,r}, K_{r,r})$ has exactly two connected components.
\end{proposition}

We now introduce what might be thought of as an extension of an $(X,Y)$-friendly swap. Proposition \ref{prop:exchangeability_components} demonstrates how this notion will be useful in the proof of Theorem \ref{thm:minimum_degree}.

\begin{definition}[{\cite[Subsection 2.4]{alon2022typical}}]
    Take $n$-vertex graphs $X$ and $Y$, a bijection $\sigma: V(X) \to V(Y)$, and distinct vertices $u, v \in V(Y)$. We say that $u$ and $v$ are \textit{\textcolor{red}{$(X,Y)$-exchangeable from $\sigma$}} if $\sigma$ and $(u \ v) \circ \sigma$ are in the same connected component of $\FS(X,Y)$. If $\Sigma$ is a sequence of $(X,Y)$-friendly swaps that transforms $\sigma$ into $(u \ v) \circ \sigma$, then we say that applying $\Sigma$ to $\sigma$ \textit{\textcolor{red}{exchanges}} $u$ and $v$.
\end{definition}

\begin{proposition}[{\cite[Proposition 2.8]{alon2022typical}}] \label{prop:exchangeability_components}
    Let $X$, $Y$, and $\Tilde{Y}$ be $n$-vertex graphs such that $Y$ is an edge-subgraph of $\Tilde{Y}$. Suppose that for every $\{u,v\} \in E(\Tilde{Y})$ and every bijection $\sigma$ satisfying $\{\sigma^{-1}(u), \sigma^{-1}(v)\} \in E(X)$, the vertices $u$ and $v$ are $(X,Y)$-exchangeable from $\sigma$. Then the number of connected components of $\FS(X, \Tilde{Y})$ is equal to the number of connected components of $\FS(X,Y)$.
\end{proposition}

Finally, we introduce a result of \cite{zhu2023evacuating}, which will be our main tool in proving Theorem \ref{thm:random_edge_subgraphs}.

\begin{definition}[\cite{milojevic2023connectivity}]
    A path $v_1, v_2, \dots, v_k$ in a graph is a \textit{\textcolor{red}{$k$-bridge}} if each edge in the path is a cut edge, $v_2, \dots, v_{k-1}$ have degree $2$ in the graph, and $v_1$ and $v_k$ do not have degree $1$.
\end{definition}

\begin{theorem}[{\cite[Theorem 1.7]{zhu2023evacuating}}] \label{thm:bipartite_fs_connectivity}
    Suppose $r \geq 5$. Let $X$ be an edge-subgraph of $K_{r,r}$. It holds that $\FS(X, K_{r,r})$ has exactly two connected components if and only if $X$ is connected, is not a cycle, and does not contain an $r$-bridge.
\end{theorem}

\section{Minimum Degree} \label{sec:minimum_degree}

Theorem \ref{thm:minimum_degree} is given by \cite[Theorem 1.10]{bangachev2022asymmetric} for even values of $r$, so we assume throughout this section (and the statements and proofs of all results within it), unless stated otherwise, that $r \geq 5$ and is odd. To begin, we establish the following generalization of \cite[Proposition 6.2]{alon2022typical}. The proof of this lemma borrows upon the ideas introduced in the proofs of \cite[Proposition 6.2]{alon2022typical} and \cite[Lemma 6.2]{bangachev2022asymmetric}. However, in order to prove this sharper statement, we will need to introduce and study a few additional cases.

\begin{lemma} \label{lem:exchangeability}
    Let $X$ and $Y$ be edge-subgraphs of $K_{r,r}$ such that $\delta(X) \geq \delta(Y)$ and $\delta(X) + \delta(Y) \geq (3r+1)/2$. Let $\sigma: V(X) \to V(Y)$ be a bijection. If $u,v$ are in different partite sets of $Y$ and are such that $\{\sigma^{-1}(u), \sigma^{-1}(v)\} \in E(X)$, then $u$ and $v$ are $(X,Y)$-exchangeable from $\sigma$. 
\end{lemma}

\begin{proof}
    We assume that $\{u,v\} \notin E(Y)$, as $u$ and $v$ are trivially $(X,Y)$-exchangeable otherwise. We also assume that $\delta(X) + \delta(Y) = (3r+1)/2$, since the lemma follows from \cite[Section 6]{bangachev2022asymmetric} otherwise. For later use, we note that the condition $r \geq \delta(X) \geq \delta(Y)$ implies that 
    \begin{align} \label{eq:delta_ineq}
        & \delta(X) \geq \lceil (3r+1)/4 \rceil,
        & \delta(Y) \geq (r+1)/2.
    \end{align}
    Let $\{A_X, B_X\}$ and $\{A_Y, B_Y\}$ respectively denote the bipartitions of $X$ and $Y$. Without loss of generality, we may assume that $u \in A_Y$ and $v \in B_Y$. Let $u' = \sigma^{-1}(u)$ and $v' = \sigma^{-1}(v)$. Our goal is to show that $\sigma$ and $(u \ v) \circ \sigma$ are in the same connected component of $\FS(X, Y)$. We may thus assume that the partite set of $X$ containing $v'$ contains at least $(r+1)/2$ elements of $\sigma^{-1}(B_Y)$, as we may simply switch the roles of $\sigma$ and $(u \ v) \circ \sigma$ otherwise. Without loss of generality, we may assume that $u' \in A_X$ and $v' \in B_X$, so that 
    \begin{align} \label{eq:wlog_assumption}
        |\sigma(B_X) \cap B_Y| \geq (r+1)/2.
    \end{align} 
    Now, let $\mu: V(X) \to V(Y)$ be a bijection such that
    \begin{itemize}
        \item $\mu$ can be obtained from $\sigma$ by applying a sequence of swaps not involving $u$ or $v$;
        \item $|\mu(A_X) \cap A_Y|$ is maximal amongst all bijections in the same connected component as $\sigma$.
    \end{itemize}
    The first condition implies that $\mu(u') = u$ and $\mu(v') = v$. Let $\Sigma$ be a sequence of $(X,Y)$-friendly swaps not involving $u$ or $v$ that transforms $\sigma$ into $\mu$. We will demonstrate that there is a sequence $\Tilde{\Sigma}$ of $(X,Y)$-friendly swaps such that applying $\Tilde{\Sigma}$ to $\mu$ exchanges $u$ and $v$. It will then follow that $\Sigma^* = \Sigma, \Tilde{\Sigma}, \rev(\Sigma)$ is a sequence of $(X,Y)$-friendly swaps such that applying $\Sigma^*$ to $\sigma$ exchanges $u$ and $v$. We break into two cases.

    \medskip
    
    \paragraph{\textbf{Case 1: We have $|\mu(A_X) \cap A_Y| = r$.}} 

    Here, we have that $\mu(A_X) = A_Y$ and $\mu(B_X) = B_Y$. Since
    \begin{align*}
        & |B_X \setminus N(u')| \leq r - \delta(X),
        & |B_Y \setminus N(u)| \leq r - \delta(Y),
    \end{align*}
    and we have that
    \begin{align*}
        (r - \delta(X)) + (r - \delta(Y)) = 2r - (3r+1)/2 = (r-1)/2 < r,
    \end{align*}
    there exists $w \in N(u)$ such that $w' = \mu^{-1}(w) \in N(u')$. Since $|N(v') \cap N(w')| \geq 2\delta(X)-r$ and $|N(v) \cap N(w)| \geq 2\delta(Y)-r$ imply
    \begin{align*}
        |A_X \setminus (N(v') \cap N(w'))| & \leq r - (2\delta(X)-r) = 2r-2\delta(X), \\
        |A_Y \setminus (N(v) \cap N(w))| & \leq r - (2\delta(Y)-r) = 2r-2\delta(Y),
    \end{align*}
    respectively, and 
    \begin{align*}
        (2r-2\delta(X)) + (2r-2\delta(Y)) = 4r - 2(\delta(X)+\delta(Y)) = 4r-(3r+1) = r-1 < r,
    \end{align*}
    there exists $x \in N(v) \cap N(w) \subseteq A_Y$ such that $x' = \mu^{-1}(x) \in N(v') \cap N(w') \subseteq A_X$. Also, $x \neq u$, since $u \notin N(v)$. We denote
    \begin{gather*}
        D_A = A_X \setminus (N(v') \cap N(w')), \quad D_B = B_X \setminus (N(u') \cap N(x')), \\
        E_A = A_X \setminus (D_A \cup \{u', x'\}) = N(v') \cap N(w') \setminus \{u', x'\}, \\
        E_B = B_X \setminus (D_B \cup \{v', w'\}) = N(u') \cap N(x') \setminus \{v', w'\}.
    \end{gather*}
    Assume that there exists $y \in N(v) \cap N(w)$ such that $\mu^{-1}(y) \in N(v') \cap N(w')$ and $y \neq x$. A visualization of these vertices and edges is given in Figure \ref{fig:bipartite_0}. Applying the sequence
    \begin{align*}
        \Tilde{\Sigma} = xv, yw, xw, uw, yw, xw, xv, yv, yw
    \end{align*}
    to $\mu$ exchanges $u$ and $v$. An entirely symmetric argument yields that if there exists $z \in N(u) \cap N(x)$ such that $\mu^{-1}(z) \in N(u') \cap N(x')$ and $z \neq w$, then we may exchange $u$ and $v$ from $\sigma$.
    \begin{figure}[ht]
        \centering
        \includegraphics[width=0.33\textwidth]{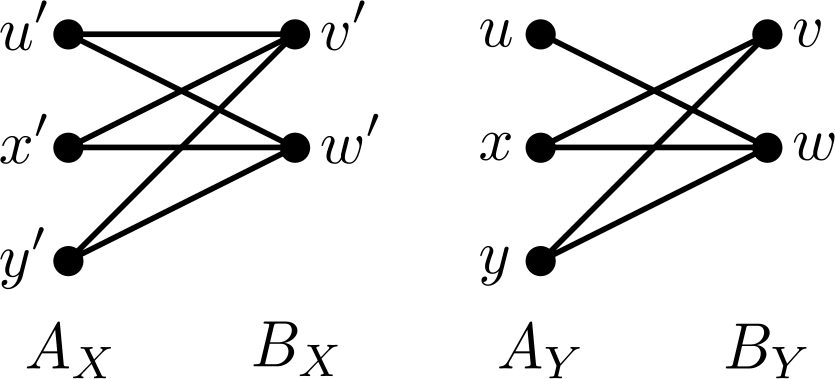}
        \caption{The vertices and edges used in Case 1.}
        \label{fig:bipartite_0}
    \end{figure}
    
    Now assume that there is no such $y$ and no such $z$. The assumption implies that
    \begin{align} \label{eq:case_1_subset_inclusions}
        & N(v) \cap N(w) \setminus \{x\} \subseteq \mu(D_A),
        & N(u) \cap N(x) \setminus \{w\} \subseteq \mu(D_B).
    \end{align}
    Furthermore, we have that
    \begin{gather}
        2\delta(Y) - r - 1 \leq |N(v) \cap N(w) \setminus \{x\}| \leq |\mu(D_A)| \leq 2r-2\delta(X), \label{eq:case_1_eq_1} \\
        2(\delta(X)+\delta(Y)) = 3r+1 \implies 2\delta(Y) - r - 1 = 2r-2\delta(X). \label{eq:case_1_eq_2}
    \end{gather}
    It follows from \eqref{eq:case_1_eq_2} that all inequalities in \eqref{eq:case_1_eq_1} are equalities, so the first subset inclusion in \eqref{eq:case_1_subset_inclusions} holds with equality. We may argue entirely analogously to study $N(u) \cap N(x) \setminus \{w\}$. Altogether,
    \begin{align} \label{eq:case_1_deductions_1}
        N(v) \cap N(w) \setminus \{x\} = \mu(D_A), \quad N(u) \cap N(x) \setminus \{w\} = \mu(D_B), \quad |D_A| = |D_B| = 2r-2\delta(X).
    \end{align}
    The final statement of \eqref{eq:case_1_deductions_1} easily implies $N(v') \cup N(w') = A_X$ and $N(u') \cup N(x') = B_X$, so that
    \begin{align} \label{eq:case_1_deductions_2}
        & N(v') \bigtriangleup N(w') = D_A,
        & N(u') \bigtriangleup N(x') = D_B.
    \end{align}
    It is also easy to see that
    \begin{align} \label{eq:E_A_and_E_B}
        & (N(v) \bigtriangleup N(w)) \setminus \{u\} \subseteq \mu(E_A),
        & (N(u) \bigtriangleup N(x)) \setminus \{v\} \subseteq \mu(E_B).
    \end{align}
    Furthermore, we have that
    \begin{gather} 
        2r-2\delta(Y) - 1 \leq |(N(v) \triangle N(w)) \setminus \{u\}| \leq |\mu(E_A)| = r-(2+|D_A|) = 2\delta(X)-r-2, \label{eq:case_1_eq_3} \\
        3r+1 = 2(\delta(X)+\delta(Y)) \implies 2r-2\delta(Y) - 1 = r - 2 - (2r-2\delta(X)). \label{eq:case_1_eq_4}
    \end{gather}
    It follows from \eqref{eq:case_1_eq_4} that all inequalities in \eqref{eq:case_1_eq_3} are equalities, so the first subset inclusion in \eqref{eq:E_A_and_E_B} holds with equality. We may argue entirely analogously to prove that the second subset inclusion in \eqref{eq:E_A_and_E_B} also holds with equality, so we have
    \begin{align} \label{eq:vts_upon_E_sets}
        & (N(v) \triangle N(w)) \setminus \{u\} = \mu(E_A),
        & (N(u) \triangle N(x)) \setminus \{v\} = \mu(E_B).
    \end{align}
    Furthermore, both of these sets have exactly $2\delta(X)-r-2$ vertices. From here, \eqref{eq:delta_ineq} implies that both sets are nonempty. We take $y' \in E_A$, and denote $y = \mu(y')$. By \eqref{eq:delta_ineq}, we have that
    \begin{align} \label{eq:neighbors_ineq}
       |\mu(E_B \cap N(y'))| & \geq (2\delta(X)-r-2)-(r-\delta(X)) = 3\delta(X) - 2r + 1 \geq 3(3r+1)/4 - 2r + 1 > 0.
    \end{align}
    Thus, there exists $z \in \mu(E_B \cap N(y'))$. Let $z' = \mu^{-1}(z)$. We break into two subcases.

    \medskip

    \paragraph{\textbf{Subcase 1.1: We have $\delta(X) < r$.}} 
    
    Here, we also have that
    \begin{align*}
       |\mu(D_A \cap N(z'))| \geq (2r-2\delta(X))-(r-\delta(X)) = r - \delta(X) > 0.
    \end{align*}
    Therefore, we may take $s' \in D_A \cap N(z')$. We let $s = \mu(s')$. Similarly, we may take $t' \in D_B \cap N(y')$. We let $t = \mu(t')$. Now, \eqref{eq:case_1_deductions_2} and \eqref{eq:vts_upon_E_sets} imply that
    \begin{align} \label{eq:subcase_1.1_more_subcases}
        |N(y) \cap \{v,w\}| = |N(z) \cap \{u,x\}| = |N(s') \cap \{v',w'\}| = |N(t') \cap \{u',x'\}| = 1.
    \end{align}
    Figure \ref{fig:bipartite_1} depicts these vertices and edges. In the order they are listed from left to right, denote the four sets with cardinality one in \eqref{eq:subcase_1.1_more_subcases} by $S_1, S_2, S_3, S_4$. There are several further subcases induced by \eqref{eq:subcase_1.1_more_subcases}. In Table \ref{tab:subcase_1.1}, we present a sequence of $(X,Y)$-friendly swaps $\Tilde{\Sigma}$ for each of these subcases such that applying $\Tilde{\Sigma}$ to $\mu$ exchanges $u$ and $v$.\footnote{In Tables \ref{tab:subcase_1.1} and \ref{tab:subcase_1.2}, we isolate those sequences of swaps $\Tilde{\Sigma}$ corresponding to the relevant subcases in the proof of Lemma \ref{lem:exchangeability}. On the right columns, we write the unique element in each set rather than the corresponding singleton set. Blanks indicate that $\Tilde{\Sigma}$ works regardless of what the unique vertex in the set is. These sequences of $(X,Y)$-friendly swaps were all found using computer assistance. In particular, all of these sequences are shortest possible.}

    \begin{table}[ht]
        \setlength\extrarowheight{2pt}
        \centering
        \begin{tabular}{ |C{0.8\linewidth}|C{0.15\linewidth}| } 
            \hline 
            $\Tilde{\Sigma}$ & $\left( S_1, S_2, S_3, S_4 \right)$ \\ 
            \hline
            $uw$, $sv$, $sw$, $xw$, $xv$, $uw$, $uz$, $sw$, $uw$, $xw$, $ut$, $xt$, $xv$, $sv$, $yv$, $xv$, $sv$, $sw$, $xw$, $uw$, $ut$, $uz$, $xv$, $xw$, $xt$, $xv$, $xw$, $yv$, $uw$, $xv$, $xw$ & $(v,u,v',u')$ \\
            $uz$, $sw$, $yv$, $sv$, $xv$, $yv$, $xw$, $sw$, $uw$, $ut$, $uz$, $xw$, $sw$, $xt$, $xw$, $ut$, $xt$, $xv$, $xw$, $xt$, $ut$, $xv$, $uz$, $xt$, $ut$, $yv$, $xv$, $sv$, $yv$, $uw$, $sw$, $sv$, $uz$, $uw$, $sw$, $xw$, $uw$, $sw$, $uz$, $xv$, $xt$, $xw$, $yv$, $sv$, $xv$, $yv$, $xt$, $xw$, $uw$, $ut$, $xt$, $uz$, $uw$, $xv$, $xw$ & $(v,u,w',u')$ \\
            $uz$, $sw$, $yv$, $sv$, $xv$, $xw$, $sw$, $sv$, $uw$, $xw$, $ut$, $xt$, $xv$, $yv$, $sv$, $xv$, $xt$, $ut$, $xw$, $uw$, $sw$, $sv$, $uz$, $uw$, $sw$, $xw$, $xv$, $sv$, $sw$ & $(v,u,w',x')$ \\
            $uw$, $xz$, $yv$, $xv$, $sv$, $sw$, $xw$, $uw$, $xz$, $ut$, $xw$, $xt$, $xz$, $xw$, $sw$, $sv$, $xv$, $yv$, $xz$, $xt$, $xw$, $ut$, $uw$, $xv$, $xw$, $xz$, $xv$, $xw$, $yv$, $xv$, $uw$, $xw$, $xz$, $xv$, $xw$, $yv$, $uw$, $xz$, $xw$, $xv$, $xz$ & $(v,x,\underline{\hspace{3mm}},\underline{\hspace{3mm}})$ \\
            $uw$, $xv$, $xw$, $yw$, $uw$, $xw$, $yw$, $xv$, $sv$, $xw$, $sw$, $yw$, $xw$, $uw$, $yw$, $sw$, $sv$, $xv$, $xw$ & $(w,u,w',u')$ \\
            $uw$, $xz$, $xw$, $yw$, $xv$, $xz$, $uw$, $xw$, $xv$, $xz$, $xw$, $uw$, $yw$, $xw$, $xz$, $xv$, $xw$ & $(w,x,\underline{\hspace{3mm}},\underline{\hspace{3mm}})$ \\
            \hline
        \end{tabular}
        \caption{Subcase 1.1. The exchangeability of $u$ and $v$ for the possibilities not listed here follows easily from left-right symmetries with listed possibilities.}
        \label{tab:subcase_1.1}
    \end{table}

    \medskip
    
    \paragraph{\textbf{Subcase 1.2: We have $\delta(X) = r$.}} 

    Here, $\delta(Y) = (r+1)/2$. We may adapt \eqref{eq:neighbors_ineq} to observe that
    \begin{align*}
        |\mu(E_B \cap N(y')) \cap N(y)| \geq (2\delta(X)-r-2)-(r-\delta(Y)) = \delta(Y)-2 > 0,
    \end{align*}
    so we may also assume that $z \in N(y)$. We take $s' \in A_X \setminus \{u', x', y'\}$ and $t' \in B_X \setminus \{v', w', z'\}$ such that $s = \mu(s')$ and $t = \mu(t')$. It follows from the assumption of this subcase that $s' \in E_A$ and $t' \in E_B$ (here, we have that $X = K_{r,r}$). Now, \eqref{eq:vts_upon_E_sets} implies that
    \begin{align} \label{eq:subcase_1.2_more_subcases}
        |N(y) \cap \{v,w\}| = |N(z) \cap \{u,x\}| = |N(s) \cap \{v,w\}| = |N(t) \cap \{u,x\}| = 1.
    \end{align}
    Figure \ref{fig:bipartite_2} depicts these vertices and edges. In the order they are listed from left to right, denote the four sets with cardinality one in \eqref{eq:subcase_1.2_more_subcases} by $T_1, T_2, T_3, T_4$. There are several further subcases induced by \eqref{eq:subcase_1.2_more_subcases}. In Table \ref{tab:subcase_1.2}, we present a sequence of $(X,Y)$-friendly swaps $\Tilde{\Sigma}$ for each of these subcases such that applying $\Tilde{\Sigma}$ to $\mu$ exchanges $u$ and $v$.
    \begin{figure}[ht]
      \centering
      \subfloat[Subcase 1.1.]{\makebox[0.4\textwidth][c]{\includegraphics[width=0.35\textwidth]{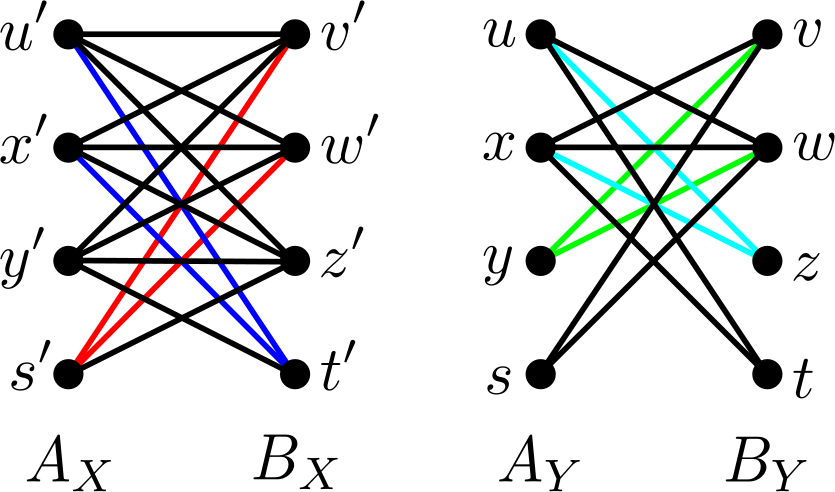}\label{fig:bipartite_1}}}
      \hfil
      \subfloat[Subcase 1.2. Here, $X = K_{r,r}$.]{\makebox[0.4\textwidth][c]{\includegraphics[width=0.15\textwidth]{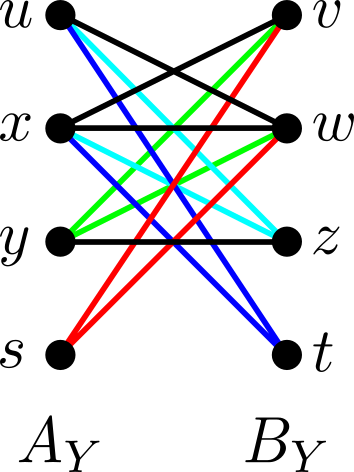}\label{fig:bipartite_2}}}
      \caption{The vertices and edges used in Subcases 1.1 and 1.2. For each subfigure, exactly one edge of a particular color is present.}
      \label{fig:case_1_subcases}
    \end{figure}

    \begin{table}[ht]
        \setlength\extrarowheight{2pt}
        \centering
        \begin{tabular}{ |C{0.8\linewidth}|C{0.15\linewidth}| } 
            \hline 
            $\Tilde{\Sigma}$ & $\left(T_1, T_2, T_3, T_4\right)$ \\ 
            \hline
            $uw$, $yz$, $sv$, $yv$, $xv$, $sv$, $xw$, $xv$, $uz$, $yz$, $yv$, $sv$, $xv$, $yv$, $yz$, $uz$, $xw$, $xv$, $yv$, $sv$, $xv$, $yz$, $xw$ & $(v,u,v,\underline{\hspace{3mm}})$ \\
            $uw$, $xt$, $xw$, $sw$, $xv$, $xt$, $uw$, $xw$, $xv$, $xt$, $xw$, $uw$, $sw$, $xw$, $xt$, $xv$, $xw$ & $(v,u,w,x)$ \\
            $xw$, $yz$, $xz$, $xv$, $yv$, $xw$, $uw$, $xz$, $xw$, $xv$, $xz$ & $(v,x,\underline{\hspace{3mm}},\underline{\hspace{3mm}})$ \\
            $uw$, $xz$, $xw$, $yw$, $yz$, $xv$, $uw$, $yw$, $xw$, $uw$, $yw$ & $(w,x,\underline{\hspace{3mm}},\underline{\hspace{3mm}})$ \\
            \hline
        \end{tabular}
        \caption{Subcase 1.2. The exchangeability of $u$ and $v$ for the possibilities not listed here follows easily from left-right symmetries with listed possibilities.}
        \label{tab:subcase_1.2}
    \end{table}

    \medskip

    \paragraph{\textbf{Case 2: We have $|\mu(A_X) \cap A_Y| < r$.}}
    
    Let $\Tilde{X} = X |_{V(X) \setminus \{u', v'\}}$ and $\Tilde{Y} = Y |_{V(Y) \setminus \{u,v\}}$. Let the partite sets of $\Tilde{X}$ corresponding to $A_X$ and $B_X$ respectively be $A_{\Tilde{X}}$ and $B_{\Tilde{X}}$, and let the partite sets of $\Tilde{Y}$ corresponding to $A_Y$ and $B_Y$ respectively be $A_{\Tilde{Y}}$ and $B_{\Tilde{Y}}$. We denote $s = |\mu(A_{\Tilde{X}}) \cap A_{\Tilde{Y}}|$. It follows from \eqref{eq:wlog_assumption} and the assumption for this case that $(r-1)/2 \leq s \leq r-1$. It also follows that
    \begin{align*}
        |\mu(A_{\Tilde{X}}) \cap B_{\Tilde{Y}}| = |\mu(B_{\Tilde{X}}) \cap A_{\Tilde{Y}}| = r-1-s, \quad |\mu(B_{\Tilde{X}}) \cap B_{\Tilde{Y}}| = s.
    \end{align*}
    There cannot exist vertices $p \in \mu(A_{\Tilde{X}}) \cap B_{\Tilde{Y}}$ and $q \in \mu(B_{\Tilde{X}}) \cap A_{\Tilde{Y}}$ satisfying $\{p,q\} \in E(Y)$ and $\{\mu(p), \mu(q)\} \in E(X)$, since the $(X,Y)$-friendly swap $pq$ would then result in a bijection contradicting the maximality of $|\mu(A_X) \cap A_Y|$. Let $m$ be the number of edges between $\mu(A_{\Tilde{X}}) \cap B_{\Tilde{Y}}$ and $\mu(B_{\Tilde{X}}) \cap A_{\Tilde{Y}}$, so that there are at most $(r-1-s)^2-m$ edges between $\mu^{-1}(\mu(A_{\Tilde{X}}) \cap B_{\Tilde{Y}}) = A_{\Tilde{X}} \cap \mu^{-1}(B_{\Tilde{Y}})$ and $\mu^{-1}(\mu(B_{\Tilde{X}}) \cap A_{\Tilde{Y}}) = B_{\Tilde{X}} \cap \mu^{-1}(A_{\Tilde{Y}})$.

    We let $a$ be a vertex in $\mu(A_{\Tilde{X}}) \cap B_{\Tilde{Y}}$ with the fewest neighbors in $\mu(B_{\Tilde{X}}) \cap A_{\Tilde{Y}}$, and let $b'$ be a vertex in $B_{\Tilde{X}} \cap \mu^{-1}(A_{\Tilde{Y}})$ with the fewest neighbors in $A_{\Tilde{X}} \cap \mu^{-1}(B_{\Tilde{Y}})$. We let $a' = \mu^{-1}(a) \in A_X$ and $b = \mu(b) \in A_{\Tilde{Y}}$. It follows that
    \begin{align}
        & |N(a) \cap (\mu(B_{\Tilde{X}}) \cap A_{\Tilde{Y}})| \leq \frac{m}{r-1-s} =: t, \label{eq:a_neighbor_set} \\
        & |N(b') \cap (A_{\Tilde{X}} \cap \mu^{-1}(B_{\Tilde{Y}}))| \leq \frac{(r-1-s)^2-m}{r-1-s} = (r-1-s) - \frac{m}{r-1-s} = r-1-s-t. \label{eq:b_neighbor_set}
    \end{align}
    We observe that
    \begin{align} 
        |N(a') \cap (B_{\Tilde{X}} \cap \mu^{-1}(B_{\Tilde{Y}}))| & = |N(a') \cap B_{\Tilde{X}}| - |N(a') \cap (B_{\Tilde{X}} \cap \mu^{-1}(A_{\Tilde{Y}}))| \nonumber \\
        & \geq (\delta(X)-1) - |B_{\Tilde{X}} \cap \mu^{-1}(A_{\Tilde{Y}})| = (\delta(X) - 1) - (r-1-s), \label{eq:case_2_eq_1}
    \end{align}
    and that \eqref{eq:case_2_eq_1} may hold with equality only if $v' \in N(a')$ and $B_{\Tilde{X}} \cap \mu^{-1}(A_{\Tilde{Y}}) \subseteq N(a')$. Similarly,
    \begin{align} \label{eq:case_2_eq_2}
        |N(b) \cap (\mu(B_{\Tilde{X}}) \cap B_{\Tilde{Y}})| & = |N(b) \cap B_{\Tilde{Y}}| - |N(b) \cap (B_{\Tilde{Y}} \cap \mu(A_{\Tilde{Y}}))| \nonumber \\
        & \geq (\delta(Y)-1) - |B_{\Tilde{Y}} \cap \mu(A_{\Tilde{Y}})| \geq (\delta(Y) - 1) - (r-1-s),
    \end{align}
    and \eqref{eq:case_2_eq_2} may hold with equality only if $v \in N(b)$ and $B_{\Tilde{Y}} \cap \mu(A_{\Tilde{X}}) \subseteq N(b)$. If both \eqref{eq:case_2_eq_1} and \eqref{eq:case_2_eq_2} held with equality, then we would have that $\{a', b'\} \in E(X)$ and $\{a, b\} \in E(Y)$, and the $(X,Y)$-friendly swap $ab$ would result in a bijection contradicting the maximality of $|\mu(A_X) \cap A_Y|$. Thus, we assume that either \eqref{eq:case_2_eq_1} or \eqref{eq:case_2_eq_2} is strict, so that
    \begin{align*}
        & |N(a') \cap (B_{\Tilde{X}} \cap \mu^{-1}(B_{\Tilde{Y}}))| + |N(b) \cap (\mu(B_{\Tilde{X}}) \cap B_{\Tilde{Y}})| \\
        & > (\delta(X) - 1) - (r-1-s) + (\delta(Y) - 1) - (r-1-s) \\
        & = (3r+1)/2 - 2r + 2s \geq (1-r)/2 + (r-1)/2 + s = |B_{\Tilde{X}} \cap \mu^{-1}(B_{\Tilde{Y}})|.
    \end{align*}
    It follows that there exists $c' \in B_{\Tilde{X}} \cap \mu^{-1}(B_{\Tilde{Y}})$ such that $\{a', c'\} \in E(X)$ and $\{b, \mu(c')\} \in E(Y)$. Arguing similarly, we may conclude that there exists $d' \in A_{\Tilde{X}} \cap \mu^{-1}(A_{\Tilde{Y}})$ such that $\{b', d'\} \in E(X)$ and $\{a, \mu(d')\} \in E(Y)$. We let $c = \mu(c')$ and $d = \mu(d')$.\footnote{Thus far, the argument in Case 2 has followed from tracing the proof of \cite[Lemma 6.2]{bangachev2022asymmetric}. However, note that what we denote by $d$ is not the same as what they denote by $d$, which we later denote instead by $w$.}

    Now, we observe that
    \begin{align*}
        |(N(a) \cap A_{\Tilde{Y}}) \cap (N(c) \cap A_{\Tilde{Y}})| \geq 2(\delta(Y)-1) - |A_{\Tilde{Y}}| = 2(\delta(Y)-1) - (r-1),
    \end{align*}
    and since $a$ has at most $t$ neighbors in $\mu(B_{\Tilde{X}}) \cap A_{\Tilde{Y}}$, the number of common neighbors in 
    \begin{align*}
        A_{\Tilde{Y}} \setminus (\mu(B_{\Tilde{X}}) \cap A_{\Tilde{Y}}) = \mu(A_{\Tilde{X}}) \cap A_{\Tilde{Y}}
    \end{align*}
    that $a$ and $c$ have satisfies
    \begin{align} \label{eq:case_2_eq_3}
        |N(a) \cap N(c) \cap (\mu(A_{\Tilde{X}}) \cap A_{\Tilde{Y}})| \geq 2(\delta(Y)-1) - (r-1) - t.
    \end{align}
    Furthermore, \eqref{eq:case_2_eq_3} may hold with equality only if $\{a, c\} \subseteq N(u)$ and \eqref{eq:a_neighbor_set} holds with equality. Similarly, we observe that
    \begin{align*}
        |(N(b') \cap A_{\Tilde{X}}) \cap (N(c') \cap A_{\Tilde{X}})| \geq 2(\delta(X)-1) - |A_{\Tilde{X}}| = 2(\delta(X)-1) - (r-1),
    \end{align*}
    and since $b'$ has at most $r-1-s-t$ neighbors in $A_{\Tilde{X}} \cap \mu^{-1}(B_{\Tilde{Y}})$, the number of common neighbors in 
    \begin{align*}
        A_{\Tilde{X}} \setminus (A_{\Tilde{X}} \cap \mu^{-1}(B_{\Tilde{Y}})) = A_{\Tilde{X}} \cap \mu^{-1}(A_{\Tilde{Y}})
    \end{align*}
    that $b'$ and $c'$ have satisfies
    \begin{align} \label{eq:case_2_eq_4}
        |N(b') \cap N(c') \cap (A_{\Tilde{X}} \cap \mu^{-1}(A_{\Tilde{Y}}))| \geq 2(\delta(X)-1) - (r-1) - (r-1-s-t).
    \end{align}
    Furthermore, \eqref{eq:case_2_eq_4} may hold with equality only if $\{b', c'\} \subseteq N(u')$ and \eqref{eq:b_neighbor_set} holds with equality. Now assume (towards a contradiction) that either \eqref{eq:case_2_eq_3} or \eqref{eq:case_2_eq_4} is strict. Then
    \begin{align*}
        & |N(a) \cap N(c) \cap (\mu(A_{\Tilde{X}}) \cap A_{\Tilde{Y}})| + |N(b') \cap N(c') \cap (A_{\Tilde{X}} \cap \mu^{-1}(A_{\Tilde{Y}}))| \\
        & > (2(\delta(Y)-1) - (r-1) - t) + (2(\delta(X)-1) - (r-1) - (r-1-s-t)) \\
        & = 2(\delta(X)+\delta(Y)) - 3r - 1 + s = (3r+1) - (3r+1) + s = |\mu(A_{\Tilde{X}}) \cap A_{\Tilde{Y}}|.
    \end{align*}
    Thus, there exists $w \in \mu(A_{\Tilde{X}}) \cap A_{\Tilde{Y}}$ such that, letting $w' = \mu^{-1}(w)$, $\{a,c\} \subseteq N(w)$ and $\{b', c'\} \subseteq N(w')$. Figure \ref{fig:bipartite_3} depicts these vertices and edges. Now, the sequence of $(X,Y)$-friendly swaps 
    \begin{align*}
        cw, aw, bc
    \end{align*}
    results in a bijection contradicting the maximality of $|\mu(A_X) \cap A_Y|$.  
    \begin{figure}[ht]
        \centering
        \includegraphics[width=0.33\textwidth]{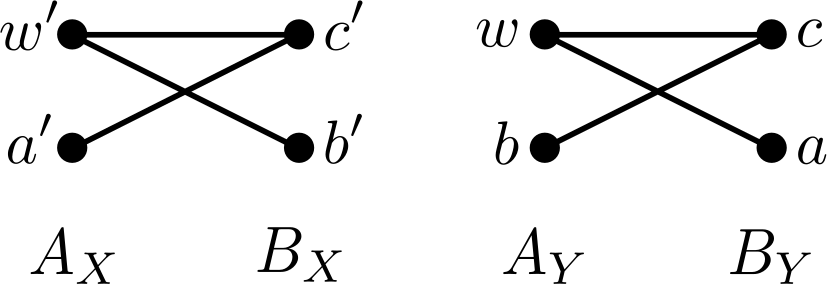}
        \caption{The vertices and edges used to raise a contradiction.}
        \label{fig:bipartite_3}
    \end{figure}
    
    Therefore, \eqref{eq:a_neighbor_set} and \eqref{eq:b_neighbor_set} both hold with equality. This implies $\{a, c\} \subseteq N(u)$ and $\{b', c'\} \subseteq N(u')$. We may argue similarly to deduce that $\{b, d\} \subseteq N(v)$ and $\{a', d'\} \subseteq N(v')$. This also implies that for any $y \in \mu(A_{\Tilde{X}}) \cap B_{\Tilde{Y}}$ and $z \in \mu(B_{\Tilde{X}}) \cap A_{\Tilde{Y}}$, exactly one of the two edges $\{\mu^{-1}(y), \mu^{-1}(z)\}$ and $\{y,z\}$ is present. In particular, exactly one of $\{a',b'\}, \{a,b\}$ is an edge, and exactly one of $\{c',d'\}, \{c,d\}$ is an edge. Figure \ref{fig:bipartite_4} depicts these vertices and edges. Here, we have that
    \begin{align} \label{eq:case_2_preliminary_cases}
        \Tilde{\Sigma} = \begin{cases}
            uc, dv, da, bc, dc, dv, uc, ua, bv & \{a', b'\} \in E(X), \{c, d\} \in E(Y) \\
            uc, dv, bv, ua, ba, bc, da, ua, bv & \{c', d'\} \in E(X), \{a, b\}  \in E(Y) \\
            uc, dv, da, bv, ba, bc, dc, da, uc, ua, bv & \{a, b\}, \{c, d\} \in E(Y) \\
        \end{cases}
    \end{align}
    is a sequence of $(X,Y)$-friendly swaps such that, when assuming the existence of the edges in a particular row of \eqref{eq:case_2_preliminary_cases}, applying the corresponding sequence to $\mu$ exchanges $u$ and $v$. The only remaining setting is that where $\{a', b'\}, \{c', d'\} \in E(X)$. Assuming this, we break into two subcases.

    \begin{figure}[ht]
        \centering
        \includegraphics[width=0.33\textwidth]{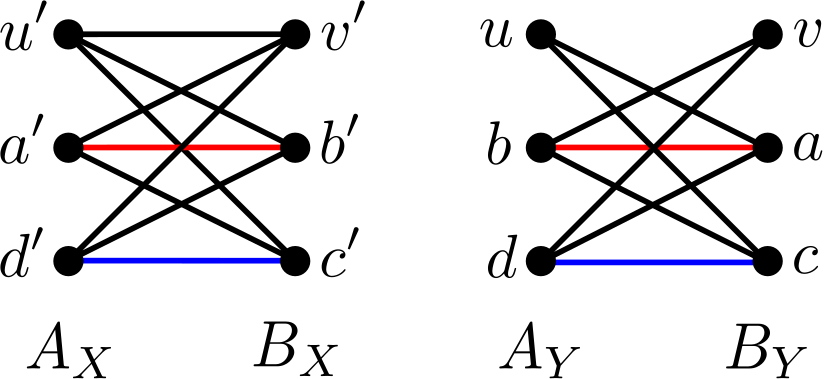}
        \caption{The vertices and edges in Case 2. Exactly one red edge and one blue edge are present.}
        \label{fig:bipartite_4}
    \end{figure}

    \medskip

    \paragraph{\textbf{Subcase 2.1: We have $s < r-2$.}} 

    The assumption of this subcase implies that there exist 
    \begin{align*}
        & w \in \mu(A_{\Tilde{X}}) \cap B_{\Tilde{Y}} \setminus \{a\},
        & x \in \mu(B_{\Tilde{X}}) \cap A_{\Tilde{Y}} \setminus \{b\}.
    \end{align*}
    Let $w' = \mu^{-1}(w)$ and $x' = \mu^{-1}(x)$. Assume that upon tracing the preceding argument with the pair $(w,x)$ playing the role of of the pair $(a,b)$, the exchangeability of $u$ and $v$ from $\mu$ still remains to be shown. Note that the argument carries over to the pair $(w,x)$ without issue, since \eqref{eq:a_neighbor_set} and \eqref{eq:b_neighbor_set} both holding with equality (which we deduced while studying the pair $(a,b)$) implies that all vertices in $\mu(A_{\Tilde{X}}) \cap B_{\Tilde{Y}}$ have the same number of neighbors in $\mu(B_{\Tilde{X}}) \cap A_{\Tilde{Y}}$, and all vertices in $B_{\Tilde{X}} \cap \mu^{-1}(A_{\Tilde{Y}})$ have the same number of neighbors in $A_{\Tilde{X}} \cap \mu^{-1}(B_{\Tilde{Y}})$. This implies the existence of 
    \begin{align*}
        & y' \in B_{\Tilde{X}} \cap \mu^{-1}(B_{\Tilde{Y}}),
        & z' \in A_{\Tilde{X}} \cap \mu^{-1}(A_{\Tilde{Y}})
    \end{align*}
    such that, letting $y = \mu(y')$ and $z = \mu(z')$, we have
    \begin{gather*}
        \{w', y'\}, \{x', z'\}, \{u', x'\}, \{u', y'\}, \{w',v'\}, \{z',v'\}, \{w', x'\}, \{y',z'\} \in E(X), \\
        \{w,z\}, \{x,y\}, \{u,w\}, \{u,y\}, \{x,v\}, \{z,v\} \in E(Y).
    \end{gather*}
    We split into three further subcases.
    \begin{figure}[ht]
      \centering
      \subfloat[Subcase 2.1.1.]{\includegraphics[width=0.32\textwidth]{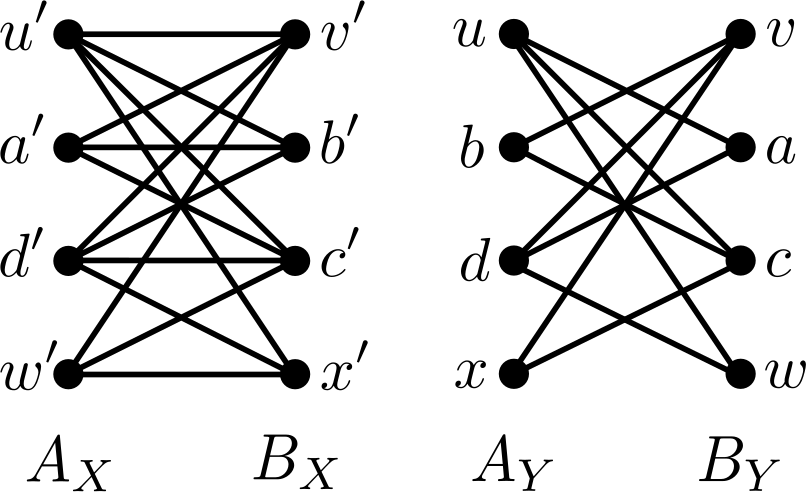}\label{fig:subcase_2.1.1}}
      \hfill
      \subfloat[Subcase 2.1.2, with $c=y$.]{\includegraphics[width=0.32\textwidth]{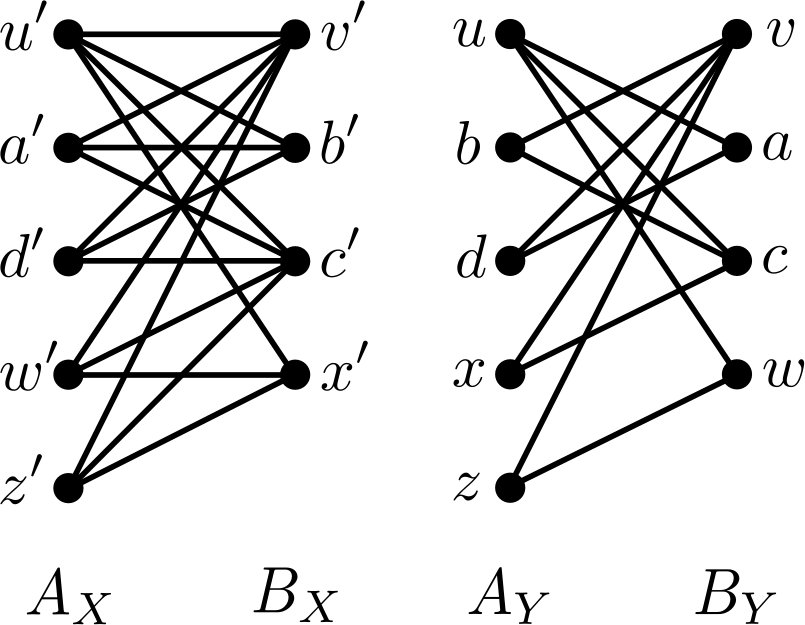}\label{fig:subcase_2.1.2}}
      \hfill
      \subfloat[Subcase 2.1.3.]{\includegraphics[width=0.32\textwidth]{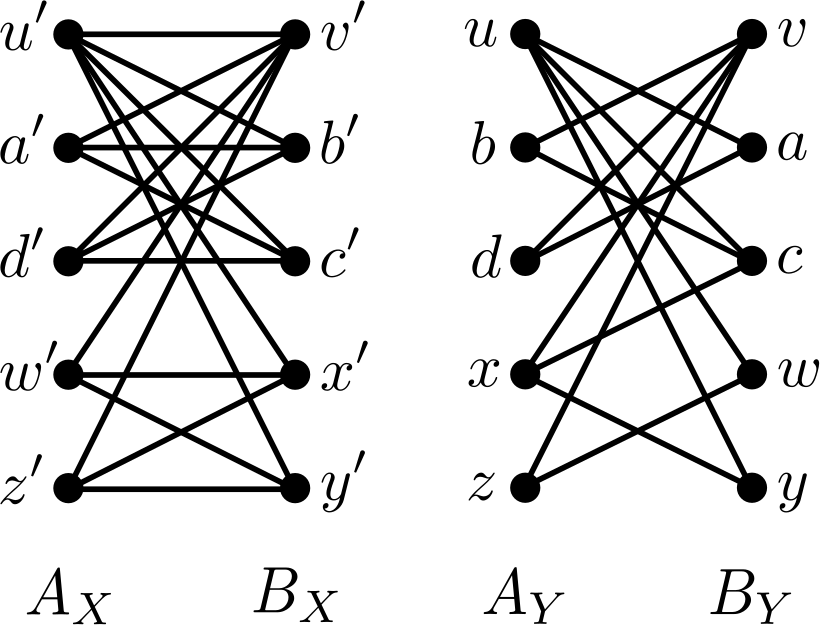}\label{fig:subcase_2.1.3}}
      \caption{The vertices and edges used in Subcase 2.1.}
      \label{fig:subcase_2.1}
    \end{figure}

    \paragraph{\textbf{Subcase 2.1.1: We have $c = y$ and $d = z$.}}

    Figure \ref{fig:subcase_2.1.1} depicts these vertices and edges. We may exchange $u$ and $v$ from $\mu$ by applying the sequence
    \begin{align*}
        \Tilde{\Sigma} = uc, dv, bv, dw, dv, xv, bv, dv, bc, bv, xc, uc, uw.
    \end{align*}

    \smallskip
    
    \paragraph{\textbf{Subcase 2.1.2: Exactly one of $c = y$ and $d = z$ holds.}}

    We assume that $c=y$ and $d \neq z$. The argument is entirely analogous if $c\neq y$ and $d=z$ holds instead. Figure \ref{fig:subcase_2.1.2} depicts these vertices and edges. We may exchange $u$ and $v$ from $\mu$ by applying the sequence
    \begin{align*}
        \Tilde{\Sigma} = uc, dv, bv, xc, xv, uw, zw, zv, bv, dv, zv, zw, uw, xv, xc, bc, bv.
    \end{align*}

    \smallskip
    
    \paragraph{\textbf{Subcase 2.1.3: We have $c \neq y$ and $d \neq z$.}}

    Figure \ref{fig:subcase_2.1.3} depicts these vertices and edges. We may exchange $u$ and $v$ from $\mu$ by applying the sequence
    \begin{align*}
        \Tilde{\Sigma} & = uc, zv, zw, xv, zv, bv, xy, ua, uw, da, zw, zv, uy, uw, uc, uy, bc, bv, \\
        & \qquad xy, xv, dv, uc, bv, zv, dv, da, bv, dv, bc, bv, zv, zw, uw, dv, da, bv, dv, bc, bv.
    \end{align*}

    \medskip

    \paragraph{\textbf{Subcase 2.2: We have $s = r-2$.}}

    The assumption for this subcase imply 
    \begin{align*}
        \mu(A_X) \cap B_Y = \{a\}, \qquad \mu(B_X) \cap A_Y = \{b\}, \qquad |\mu(A_X) \cap A_Y| = |\mu(B_X) \cap B_Y| = r-1.
    \end{align*}
    We split into two further subcases.

    \smallskip
    
    \paragraph{\textbf{Subcase 2.2.1: We have $\delta(X) < r$.}}
    
    We observe that
    \begin{gather*}
        |A_Y \setminus (N(v) \cap N(c))| \leq 2r-2\delta(Y), \quad |A_Y \setminus N(d)| \leq r-\delta(Y), \\
        |A_X \setminus (N(a') \cap N(w'))| \leq 2r-2\delta(X), \quad |A_X \setminus N(c')| \leq r-\delta(X).
    \end{gather*}
    Since we have that
    \begin{align*}
        2(r-\delta(Y)) + (r-\delta(X)) = 3r - 2(\delta(X)+\delta(Y)) + \delta(X) = \delta(X)-1 < r-1,
    \end{align*}
    there exists $w \in N(v) \cap N(c)$ such that $w' = \mu^{-1}(w) \in N(c')$. Since $u \notin N(v)$ and $d \notin N(c)$, we have $w \notin \{u, d\}$. Similarly, there exists $x \in N(u) \cap N(d) \setminus \{v, c\}$ such that $x' = \mu^{-1}(x) \in N(d')$. Also, since 
    \begin{align*}
        2(r-\delta(X)) + (r-\delta(Y)) = 3r - 2(\delta(X)+\delta(Y)) + \delta(Y) = \delta(Y)-1 \leq \delta(X)-1 < r-1,
    \end{align*}
    there exists $y \in N(d)$ such that $y' = \mu^{-1}(y) \in N(a') \cap N(w')$. Figure \ref{fig:subcase_2.2.1} depicts these vertices and edges. If $y \in \{v,x\}$, we may exchange $u$ and $v$ from $\mu$ by applying the sequence
    \begin{align*}
        \Tilde{\Sigma} = \begin{cases}
            dv, wc, bv, da, dv, ua, wv, da, dv, uc, bc, bv, wv, dv, wc, uc, da, wv, dv, wc, bv & y = v, \\
            uc, dx, bc, ua, uc, ux, ua, uc, da, ua, dv, da, dx, dv, ux, ua, bv & y = x.
        \end{cases}
    \end{align*}
    If $y \notin \{v,x\}$, we may exchange $u$ and $v$ from $\mu$ by applying the sequence
    \begin{align*}
        \Tilde{\Sigma} & = uc, bc, ua, ad, bv, dx, ux, dy, ua, uc, ux, ua, wc, da, dx, dv, wv, bv, dv, dx, \\
        & \qquad da, ua, dy, da, dv, bv, wv, wc, bc, uc, dx, bv, dv, ua, da, dx, dv, da, ux, uc, ua.
    \end{align*}

    \begin{figure}[ht]
      \centering
      \subfloat[Case where $y = v$.]{\includegraphics[width=0.32\textwidth]{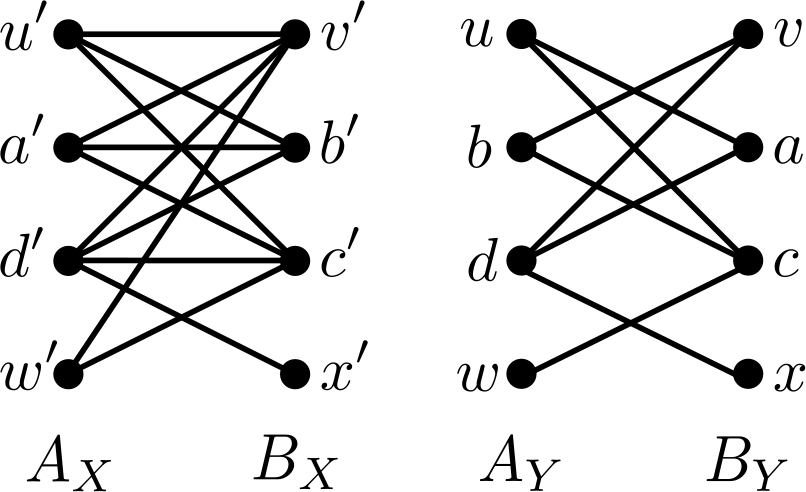}\label{fig:subcase_2.2.1_v}}
      \hfill
      \subfloat[Case where $y = x$.]{\includegraphics[width=0.32\textwidth]{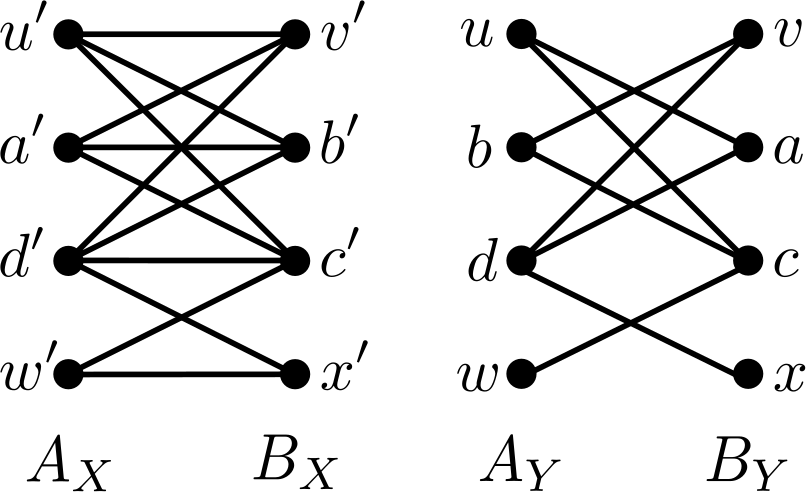}\label{fig:subcase_2.2.1_x}}
      \hfill
      \subfloat[Case where $y \notin \{v,x\}$.]{\includegraphics[width=0.32\textwidth]{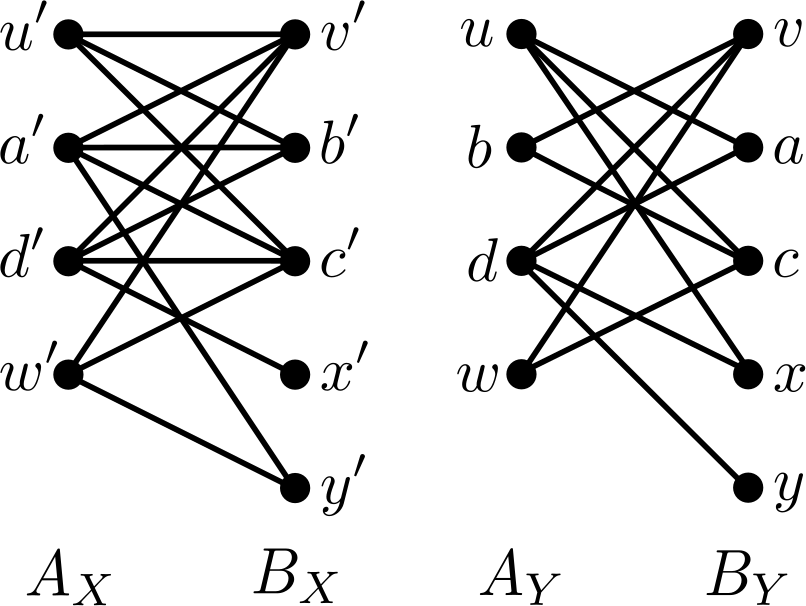}\label{fig:subcase_2.2.1_neither}}
      \caption{The vertices and edges used in Subcase 2.2.1.}
      \label{fig:subcase_2.2.1}
    \end{figure}

    \paragraph{\textbf{Subcase 2.2.2: We have $\delta(X) = r$.}} Here, $X = K_{r,r}$ and $\delta(Y) = (r+1)/2$. If there existed $w \in N(v) \cap N(a)$ such that $w' = \mu^{-1}(w) \in A_X \setminus \{u', a', d'\}$, then we may exchange $u$ and $v$ from $\mu$ by applying the sequence
    \begin{align*}
        \Tilde{\Sigma} = uc, dv, bv, ua, wa, da, ua, dv, wv, bv, dv, bc, bv.
    \end{align*}
    Figure \ref{fig:bipartite_9a} depicts these vertices and edges. We can argue the exchangeability of $u$ and $v$ from $\mu$ similarly if we replace $N(v) \cap N(a)$ in the preceding argument with one of 
    \begin{align*}
        N(v) \cap N(c), \qquad N(a) \cap N(c).
    \end{align*}
    In an analogous manner, we can argue the exchangeability of $u$ and $v$ from $\mu$ if we replace the condition $\mu^{-1}(w) \in A_X \setminus \{u', a', d'\}$ with the condition $w' = \mu^{-1}(w) \in B_X \setminus \{v', b', c'\}$ and also replace $N(v) \cap N(a)$ in the preceding argument with one of 
    \begin{align*}
        N(u) \cap N(b), \qquad N(u) \cap N(d), \qquad N(b) \cap N(d).
    \end{align*}
    Therefore, we may further assume that none of these six conditions hold. It follows from $|N(a)| \geq (r+1)/2 > 2$ that there exists $w \in N(a) \setminus \{u,d\}$. Since $b \notin N(a)$, we have that $w' = \mu^{-1}(w) \in A_X \setminus \{u',a',d'\}$. It now follows from 
    \begin{align*}
        |B_Y \setminus N(b)| \leq r-\delta(Y), \quad |B_Y \setminus N(w)| \leq r-\delta(Y), \quad (r-\delta(Y)) + (r-\delta(Y)) = r-1 < r
    \end{align*}
    that there exists $x \in N(b) \cap N(w)$. Furthermore, it follows from our most recent assumption and the fact that $a \notin N(b)$ that we have $x' = \mu^{-1}(x) \in B_X \setminus \{v',b',c'\}$. It now follows from 
    \begin{align*}
        |N(x) \setminus \{b,w\}| \geq (r+1)/2 - 2 > 0
    \end{align*}
    and our most recent assumption that there exists $y \in N(x) \setminus \{b,w\}$. Since $y \neq b$, we have that $y' = \mu^{-1}(y) \in A_X$. Similarly, there exists $z \in N(w) \setminus \{a,x\}$ such that $z' = \mu^{-1}(w) \in B_X$. Figure \ref{fig:bipartite_9b} depicts these vertices and edges. We may now exchange $u$ and $v$ from $\mu$ by applying the sequence
    \begin{align*}
        \Tilde{\Sigma} = uc, wz, ua, da, wa, wx, bx, bv, bc, dv, uc, bv, bc, bx, wx, wa, wz.
    \end{align*}
    \begin{figure}[ht]
      \centering
      \subfloat[Assuming $w$.]{\makebox[0.4\textwidth][c]{\includegraphics[width=0.14\textwidth]{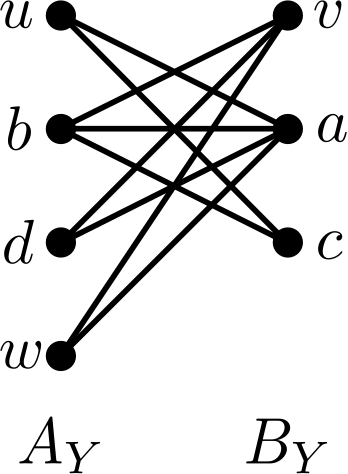}}\label{fig:bipartite_9a}}
      \hfil
      \subfloat[Assuming more.]{\makebox[0.4\textwidth][c]{\includegraphics[width=0.14\textwidth]{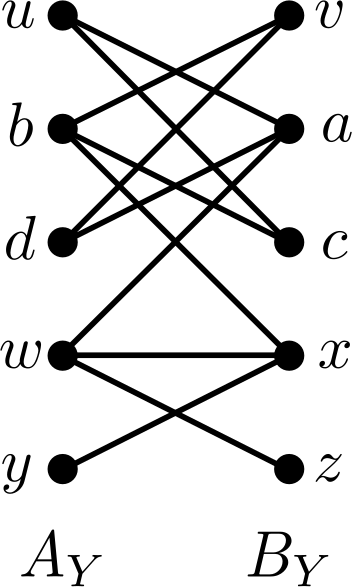}\label{fig:bipartite_9b}}}
      \caption{The vertices and edges used in Subcase 2.2.2.}
      \label{fig:subcase_2.2.2}
    \end{figure}
    
    This completes the proof of the lemma.
\end{proof}

\begin{proof}[Proof of Theorem \ref{thm:minimum_degree}]
    Suppose $X$ and $Y$ are edge-subgraphs of $K_{r,r}$ such that $\delta(X) \geq \delta(Y)$ and $\delta(X) + \delta(Y) \geq (3r+1)/2$. Lemma \ref{lem:exchangeability} implies that the hypothesis of Proposition \ref{prop:exchangeability_components} is satisfied with $\Tilde{Y} = K_{r,r}$, so it follows that $\FS(X,Y)$ and $\FS(X, K_{r,r})$ have the same number of connected components. Since $K_{r,r}$ and $X$ are both edge-subgraphs of $K_{r,r}$ with $\delta(K_{r,r}) \geq \delta(X)$ and $\delta(K_{r,r}) + \delta(X) \geq (3r+1)/2$, Lemma \ref{lem:exchangeability} implies that the hypothesis of Proposition \ref{prop:exchangeability_components} is satisfied with the pair $(K_{r,r}, X)$ playing the role of $(X,Y)$ and with $\Tilde{Y} = K_{r,r}$, so $\FS(K_{r,r}, X)$ and $\FS(K_{r,r}, K_{r,r})$ have the same number of connected components. Propositions \ref{prop:basic_properties} and \ref{prop:two_complete_bipartite} respectively imply that $\FS(X, K_{r,r}) \cong \FS(K_{r,r}, X)$ and that $\FS(K_{r,r}, K_{r,r})$ has two connected components. Altogether, it follows that $\FS(X,Y)$ also has two connected components. The theorem follows for edge-subgraphs $X$ and $Y$ of $K_{r,r}$ such that $\delta(X) < \delta(Y)$ by invoking Proposition \ref{prop:basic_properties}.
\end{proof}

Since $2\lceil (3r+1)/4 \rceil \geq \lfloor 3r/2 \rfloor + 1$, Corollary \ref{cor:symmetric_minimum_degree} follows immediately from Theorem \ref{thm:minimum_degree}. We have confirmed via a computer check that for a $2$-regular bipartite graph $Y$ whose partite sets each have three vertices, $\FS(K_{3,3}, Y)$ has exactly $12$ connected components. Corollary \ref{cor:complete_bipartite_minimum_degree} follows from this observation together with Theorem \ref{thm:minimum_degree}.

\section{Random Edge-Subgraphs} \label{sec:random_edge_subgraphs}

We will invoke the following result in the proof of Theorem \ref{thm:random_edge_subgraphs}.
\begin{proposition} \label{prop:whp_conditions}
    Let $X \in \mathcal G(K_{r,r},p)$. Then $X$ is disconnected with high probability if $p = \frac{\log r - \omega(r)}{r}$, and $X$ is connected with high probability if $p = \frac{\log r + \omega(r)}{r}$.
\end{proposition}
The latter statement of Proposition \ref{prop:whp_conditions} is \cite[Exercise 4.3.8]{frieze2016introduction}. Proposition \ref{prop:whp_conditions} follows from standard techniques in random graph theory, so we have deferred its proof to Appendix \ref{sec:exercise_solution}.

\begin{proof}[Proof of Theorem \ref{thm:random_edge_subgraphs}]
    The first part of Theorem \ref{thm:random_edge_subgraphs} follows immediately from Proposition \ref{prop:whp_conditions}, since $\FS(X,K_{r,r})$ has more than two connected components whenever $r \geq 2$ and $X$ is disconnected. We assume for the rest of the argument that $r$ is large, and that $p = \frac{\log r + \omega(r)}{r}$. All asymptotic notation will be with respect to $r$. If $p > 1/2$, it is clear that $X$ has no $r$-bridge with high probability, since $\delta(X) \geq r/2$ with high probability; this is straightforward to prove by the Chernoff bound together with a union bound over the vertices of $K_{r,r}$, for instance. Thus, we also assume that $\omega(r)$ is such that $p \leq 1/2$. We let
    \begin{itemize}
        \item $\mathcal C_1(r)$ denote the collection of all edge-subgraphs of $K_{r,r}$ with a path $v_1, \dots, v_r$ such that $v_2, \dots, v_{r-1}$ have degree $2$ in $X$;
        \item $\mathcal C_2(r)$ denote the collection of all edge-subgraphs of $K_{r,r}$ which have a connected component that is a path on a number of vertices in $\{r-2, r-1, \dots, 2r-2\}$;
        \item $\mathcal C_2(r,k)$ denote the subcollection of $\mathcal C_2(r)$ consisting of such edge-subgraphs whose largest path component is on $k$ vertices, so that $\mathcal C_2(r) = \bigsqcup_{k=r-2}^{2r-2} \mathcal C_2(r,k)$;
        \item $\mathcal C_3(r)$ denote the collection of all edge-subgraphs of $K_{r,r}$ which have a connected component that is a tree on a number of vertices in $\{r-2, r-1, \dots, 2r-2\}$;
        \item $\mathcal C_3(r,k)$ denote the subcollection of $\mathcal C_3(r)$ consisting of such edge-subgraphs whose largest tree component is on $k$ vertices, so that $\mathcal C_3(r) = \bigsqcup_{k=r-2}^{2r-2} \mathcal C_3(r,k)$.
    \end{itemize}
    The event $X \in \mathcal C_1(r)$ contains the event that $X$ is a cycle and the event that $X$ contains an $r$-bridge. Take $X \in \mathcal C_1(r)$, with $k$-bridge $v_1, \dots, v_k$ such that $k$ is maximal; in particular, $r \leq k \leq 2r$. For such a graph $X$, we may remove the edges $\{v_1, v_2\}$ and $\{v_{k-1}, v_k\}$ from $X$ to form an edge-subgraph with a component that is a path on $k-2$ vertices, and thus lies in $\mathcal C_2(r)$. This operation gives rise to a map $\varphi: \mathcal C_1(r) \to \mathcal C_2(r)$ such that any edge-subgraph $Y \in \mathcal C_2(r)$ satisfies
    \begin{align*}
        |\varphi^{-1}(Y)| \leq 2(r+2)^2.
    \end{align*}
    Indeed, $Y$ has at most two path components on at least $r-2$ vertices. Furthermore, since $p \leq 1/2$, any edge-subgraph $X \in \mathcal C_1(r)$ is no more likely in $\mathcal G(K_{r,r},p)$ than its corresponding edge-subgraph $\varphi(X) \in \mathcal C_2(r)$, which was reached by deleting edges from $X$. We now fix some $k \in \{r-2, \dots, 2r-2\}$. The number of paths on $k$ vertices is $k!/2$, and by Cayley's formula, the number of trees on $k$ vertices is $k^{k-2}$. We break into cases based on the value of $k$.

    \smallskip
    
    \paragraph{\textbf{Case 1: We have $k > r$.}} Since we are concerned with edge-subgraphs of $K_{r,r}$, which has exactly $2r$ vertices, all graphs in $\mathcal C_2(r,k)$ and $\mathcal C_3(r,k)$ necessarily have exactly one tree component on $k$ vertices. Thus, we may partition both of the collections $\mathcal C_2(r,k)$ and $\mathcal C_3(r,k)$ based on the vertex set for the unique tree component on $k$ vertices and the adjacencies amongst the remaining $2r-k$ vertices of $K_{r,r}$. By fixing some block of this partition, comparing the sizes of the corresponding subsets of $\mathcal C_2(r,k)$ and $\mathcal C_3(r,k)$ that are in this block, and summing over all blocks in this partition, it follows that uniformly over such values of $k$,\footnote{We say that $A \ll B$ if $A/B \to 0$ as $r \to \infty$. We say that $A \lesssim B$ if there exists a constant $C > 0$ such that $A \leq CB$ for all large $r$.}
    \begin{align*}
        \frac{|\mathcal C_2(r,k)|}{|\mathcal C_3(r,k)|} = \frac{k!/2}{k^{k-2}} \ll 1/k^2 \lesssim 1/r^2,
    \end{align*}
    where the $\ll$ is from (for example) Stirling's approximation. Since all graphs in $\mathcal C_2(r,k)$ and $\mathcal C_3(r,k)$ in a given block of this partition have the same number of edges, they all have the same probability of being realized under $\mathcal G(K_{r,r},p)$. By summing over all blocks of the partition, we conclude that
    \begin{align*}
        \Pr[X \in \mathcal C_2(r,k)] \ll (1/r^2) \Pr[X \in \mathcal C_3(r,k)].
    \end{align*}
    
    \paragraph{\textbf{Case 2: We have $k \leq r$.}} 
    
    Here, $X \in \mathcal C_2(r,k)$ may have up to two tree components of size at least $k$. We thus partition $\mathcal C_2(r,k)$ based on the vertex sets for the tree components on at least $k$ vertices, at least one of which is a path, and the adjacencies amongst the remaining vertices of $K_{r,r}$. By proceeding similarly as in Case 1, we may deduce that
    \begin{align*}
        \Pr[X \in \mathcal C_2(r,k)] \ll (1/r^2) \sum_{i=r-2}^{r+2} \Pr[X \in \mathcal C_3(r,i)].
    \end{align*}
    
    Altogether, letting $X \in \mathcal G(K_{r,r},p)$, it now follows that
    \begin{align*}
        \Pr[X \in \mathcal C_1(r)] & \lesssim r^2 \Pr[X \in \mathcal C_2(r)] = r^2 \sum_{k=r-2}^{2r-2} \Pr[X \in \mathcal C_2(r,k)] \ll r^2 \cdot (1/r^2) \sum_{k=r-2}^{2r-2} \Pr[X \in \mathcal C_3(r,k)] \\
        & = \Pr[X \in \mathcal C_3(r)] \leq \Pr[ X \text{ is disconnected} ].
    \end{align*}
    The desired result now follows easily from Theorem \ref{thm:bipartite_fs_connectivity} and Proposition \ref{prop:whp_conditions}.
\end{proof}

\section{Future Directions} \label{sec:future_directions}

The main results of this article provide satisfying conclusions to a number of problems with formerly partial solutions. We now conclude by opening some other doors, proposing several future new (to the best of our knowledge) problems related to this content of this article.

\subsection{Cutoff Phenomena}

Theorem \ref{thm:minimum_degree} confirms that in the bipartite setting, the cutoff for avoiding isolated vertices is the same as the cutoff for $\FS(X,Y)$ to have the smallest possible number of connected components. This is in stark contrast to what happens in the non-bipartite variant of the problem: as noted in the discussion around \cite[Theorem 1.3]{alon2022typical}, if $\delta(X) \geq n/2$ and $\delta(Y) \geq n/2$, then $\FS(X,Y)$ cannot have an isolated vertex. An interesting inquiry results from asking under what regimes on $X$ and $Y$ the two cutoffs are equal, and whether the result for edge-subgraphs of $K_{r,r}$ is really a special case of a more general phenomenon. In particular, we notice that $K_{r,r}$ is an $r$-regular graph on $2r$ vertices, so that if there were a discrepancy between the two cutoffs, then it could not have been due to the obstruction in the non-bipartite case that we mentioned above. We thus wonder if this obstruction is actually the central part of the story. As a first step towards a deeper investigation, we pose the following Question \ref{ques:edge_subgraphs}.

\begin{question} \label{ques:edge_subgraphs}
    Let $X$ and $Y$ be edge-subgraphs of an $n$-vertex graph $G$ such that $\Delta(G) \leq n/2$ and $\FS(G, G)$ has no isolated vertex. Is it always the case that the cutoff for $\FS(X,Y)$ to avoid isolated vertices is the same as the cutoff for $\FS(X,Y)$ to have the same number of connected components as $\FS(G, G)$?
\end{question}

On a related note, we remind the reader of \cite[Question 7.5]{alon2022typical}, which asks whether or not there is a stopping time result for friends-and-strangers graphs that is analogous to the corresponding classical phenomenon for the binomial random graph.

\subsection{Asymmetric Generalizations}

We may study generalizations of the setups under which the main results of this article, and those of \cite{alon2022typical, bangachev2022asymmetric, zhu2023evacuating}, were obtained. In one direction, we note that Theorem \ref{thm:bipartite_fs_connectivity} is really a special case of \cite[Theorem 1.7]{zhu2023evacuating}, which applies for imbalanced bipartite graphs. In this spirit, one natural setup results from investigating the connectivity of $\FS(X,Y)$, where we take (for values $a,b < n$) $X$ to be an edge-subgraph of $K_{a,n-a}$ and $Y$ to be an edge-subgraph of $K_{b,n-b}$. Here, we pose the following Problem \ref{prob:asymmetric_generalization_extremal}.
\begin{problem} \label{prob:asymmetric_generalization_extremal}
    Fix values $a,b < n$. Let $X$ be an edge-subgraph of $K_{a,n-a}$ and $Y$ be an edge-subgraph of $K_{b,n-b}$. Find conditions on $\delta(X)$ and $\delta(Y)$ which guarantee that $\FS(X,Y)$ has exactly two connected components.
\end{problem}
In another direction, one way to interpret the setup under which Theorem \ref{thm:random_edge_subgraphs} holds is by assuming that we draw $X \sim \mathcal G(K_{r,r}, 1)$ and $Y \sim \mathcal G(K_{r,r}, p)$. We may asymmetrize the problem by letting $X \sim \mathcal G(K_{r,r}, p_1)$ and $Y \sim \mathcal G(K_{r,r}, p_2)$ for two values $p_1, p_2 \in [0,1]$. Notably, this model also encapsulates the results of \cite{alon2022typical, milojevic2023connectivity}. We may then consider the following generalization of \cite[Question 7.6]{alon2022typical}.
\begin{question} \label{ques:asymmetric_generalization_prob}
    Under what conditions on $p_1$ and $p_2$ will $\FS(X,Y)$ have two connected components with high probability? Under what conditions on $p_1$ and $p_2$ will $\FS(X,Y)$ have more than two connected components with high probability?
\end{question}
Since the structural results of \cite{zhu2023evacuating} are concerned strictly with the setting in which $Y$ is fixed to be a complete multipartite graph, progress on Problem \ref{prob:asymmetric_generalization_extremal} and Question \ref{ques:asymmetric_generalization_prob} will likely require many novel ideas. 

\subsection{Number of Connected Components}

In this paper, we derived conditions which concerned whether or not $\FS(X,Y)$ achieved the minimum possible number of connected components. We may instead ask for a more granular understanding of the number of connected components of $\FS(X,Y)$. In this direction, we pose the following probabilistic problem, which gives a bipartite analogue of \cite[Problem 7.9]{alon2022typical}.
\begin{problem}
    Let $X$ and $Y$ be independent random graphs drawn from $G(K_{r,r},p)$. Obtain estimates for the expected number of connected components of $\FS(X,K_{r,r})$, and for the expected number of connected components of $\FS(X,Y)$.
\end{problem}

\section*{Acknowledgments}

I would like to sincerely thank two anonymous referees for thoroughly reading over a prior draft of this article, catching a number of mistakes, for providing constructive feedback, and for suggesting many intriguing directions for future research (which have been implemented in Section \ref{sec:future_directions}).

\begin{appendices}

\section{Proof of Proposition \ref{prop:whp_conditions}} \label{sec:exercise_solution}

\begin{proof}
    For $k \in [2r]$, let $X_k$ denote the number of connected components of $X$ on $k$ vertices. For each $v \in V(X)$, let $I_v$ be the indicator random variable corresponding to the event that $v$ is an isolated vertex in $X$, so that $X_1 = \sum_{v \in V(X)} I_v$. We have
    \begin{align}
        \E[X_1] & = 2r(1-p)^r = 2e^{\log r + r\log(1-p)} = 2e^{\log r - r(p+O(p^2))}, \label{eq:isolated_vx_expectation} \\
        \E[X_1^2] & = \mathop{\sum_{u, v \in V(X),}}_{u \neq v} \Pr[I_uI_v = 1] + \sum_{v \in V(X)} \Pr[I_v = 1] = \mathop{\sum_{u, v \in V(X),}}_{u \neq v} \Pr[I_uI_v = 1] + \E[X_1] \nonumber \\
        & \geq 2r(2r-1)(1-p)^{2r} + \E[X_1] \gtrsim (2r)^2(1-p)^{2r} + \E[X_1] = \E[X_1]^2 + \E[X_1]. \nonumber
    \end{align}
    Towards proving the first statement, we begin by assuming that $p = \frac{\log r - \omega(r)}{r}$. We additionally assume that $\omega(r) \leq \log r$, so that $p \geq 0$. Here, we may further \eqref{eq:isolated_vx_expectation} to get
    \begin{align*}
        \E[X_1] & = e^{\omega(r) + O((\log r)^2/r)} \gg 1.
    \end{align*}
    Since we have that
    \begin{align*}
        \Pr[X_1 > 0] \geq \Pr\left[|X_1 - \E[X_1]| \geq \E[X_1]\right] \geq 1 - \frac{\Var[X_1]}{\E[X_1]^2} \geq 1 - \frac{1}{\E[X_1]} = 1-o(1),
    \end{align*}
    it holds with high probability that $X$ contains an isolated vertex, which implies the first statement. Now assume that $p = \frac{\log r + \omega(r)}{r}$ and that $\omega(r)$ is such that $p \leq 1$. It is clear that the probability that $X$ is connected increases in $p$, so it suffices to prove the desired result when $\omega(r) \ll \log r$: henceforth, we assume this. Here, we may further \eqref{eq:isolated_vx_expectation} to get
    \begin{align} \label{eq:isolated_vx_expectation_supercritical}
        \E[X_1] & = e^{-\omega(r) + O((\log r + \omega(r))^2/r)} \ll 1.
    \end{align} 
    If $X$ is disconnected, it must have a component with at most $r$ vertices. We now closely follow the proof of \cite[Theorem 4.1]{frieze2016introduction}. We have that
    \begin{align*}
        \Pr\left[ X_1 > 0 \right] & \leq \Pr\left[X \text{ is disconnected}\right] \leq \Pr\left[ X_1 > 0 \right] + \sum_{k=2}^r \Pr\left[ X_k > 0 \right].
    \end{align*}
    Now, where the latter inequality in the first line bounds $\E[X_k]$ by the number of spanning trees for components on $k$ vertices and invokes Cayley's theorem,
    \begin{align*}
        \sum_{k=2}^r \Pr\left[ X_k > 0 \right] & \leq \sum_{k=2}^r \E\left[ X_k \right] \leq \sum_{k=2}^r \binom{2r}{k} k^{k-2}p^{k-1}(1-p)^{k(2r-k)} \\
        & \leq \sum_{k=2}^r \left( \frac{2re}{k} \right)^k k^{k-2}\left(\frac{\log r + \omega(r)}{r} \right)^{k-1}e^{-k(2r-k)\frac{\log r + \omega(r)}{r}} \\
        & \lesssim \sum_{k=2}^r ( 2re )^k \left(\frac{\log r}{r} \right)^{k-1}e^{-k(2r-k)\frac{\log r + \omega(r)}{r}} \\
        & \lesssim \sum_{k=2}^{10} (re)^k \left(\frac{\log r}{r} \right)^{k-1}e^{k(5-r)\frac{\log r + \omega(r)}{r}} + \sum_{k=11}^r ( 2re )^k\left(\frac{\log r}{r} \right)^{k-1}e^{-k(\log r + \omega(r))/2} \\
        & \lesssim \sum_{k=2}^{10} e^{-k\omega(r)}\left(\frac{\log r}{r} \right)^{k-1} + \sum_{k=11}^r \frac{r}{4}\left(\frac{e^{1-\omega(r)/2}\log r}{(r/4)^{1/2}} \right)^k = o(1) + \sum_{k=11}^r (r/4)^{1+o(1)-k/2} \ll 1.
    \end{align*}
    It thus follows that
    \begin{align*}
        \Pr\left[X \text{ is connected}\right] = \Pr\left[ X_1 = 0 \right] + o(1) = 1-o(1),
    \end{align*}
    since it follows from Markov's inequality and \eqref{eq:isolated_vx_expectation_supercritical} that $\Pr[X_1 \geq 1] \ll 1$.
\end{proof}

\end{appendices}

\section*{References}

\printbibliography[heading=none]

\end{document}